\documentclass[oneside,english]{amsart}
\usepackage[T1]{fontenc}
\usepackage[latin9]{inputenc}
\usepackage{amstext}
\usepackage{amsthm}
\usepackage{amssymb}

\makeatletter
\numberwithin{equation}{section}
\numberwithin{figure}{section}
\theoremstyle{plain}
\newtheorem{thm}{\protect\theoremname}[section]
\theoremstyle{definition}
\newtheorem{defn}[thm]{\protect\definitionname}
\theoremstyle{plain}
\newtheorem{lem}[thm]{\protect\lemmaname}

\makeatother

\usepackage{babel}
\providecommand{\definitionname}{Definition}
\providecommand{\lemmaname}{Lemma}
\providecommand{\theoremname}{Theorem}

\begin{document}

\title{Abundance of Matrices In Gaussian Integers}

\author{Aninda Chakraborty}
\begin{abstract}
In \cite{key-43}, N. Hindman, I. Leader and D. Strauss proved the
abundance for a matrix with rational entries. In this paper we proved
it for the ring of Gaussian integers. We showed the result when the
matrix is taken with entries from $\mathbb{Q}\left[i\right]$. The
main obstacle is in the field of complex numbers, no linear order
relation exists. We overcome that in a tactful way.
\end{abstract}

\maketitle

\section{Introduction}

Some problems in Ramsey Theory can be seen as follows:

Let $M$ be any $r\times s$ matrix, where $r,s\in\mathbb{N}$and
all entries of $M$ are from $\mathbb{N}$. Then, could it be possible
for each and every partition (finite) of $\mathbb{N}$ , one can find
some $\vec{u}\in\mathbb{N}^{s}$ so that all entries from $M\vec{u}$
are in one cell from the partition?

At first, take van der Waerden\textquoteright s Theorem. Then the
elements from the arithmetic progression $\{a,a+d,\ldots,a+ld\}$
is nothing but entries from

\[
\left(\begin{array}{cc}
1 & 0\\
1 & 1\\
\vdots & \vdots\\
1 & l
\end{array}\right)\cdot\left(\begin{array}{c}
a\\
d
\end{array}\right).
\]

So, we have this typical definition:
\begin{defn}
Let $\left(T,+\right)$ be a commutative monoid. Suppose $r,s\in\mathbb{N}$
and $M$ be any $r\times s$ matrix which has entries in $\mathbb{N}\cup\left\{ 0\right\} $.
Then $M$ is called image partition regular on $T$ (symbolised as
$IPR/T$) if and only if for each $k\in\mathbb{N}$ and $T=\bigcup_{j=1}^{k}D_{i}$,
then there must exist $j\in\left\{ 1,2,\ldots,k\right\} $ and $\vec{u}\in\left(T\setminus\left\{ 0\right\} \right)^{s}$
so that $M\vec{u}\in D_{j}^{r}$. 
\end{defn}

Let $\left(T,+\right)$ is a commutative monoid, $t\in T$ be any
element and $m\in\mathbb{N}$. Then we denote the $m$ times sum of
$t$ with itself by $mt$. 

In 1993, N. Hindman and I. Leader has introduced image partition regularity
of matrices in \cite{key-42}. Next in \cite{key-43} authors provided
some new characterizations and consequences of image partition regularity.
One of there characterizations is that the image partition regular
matrices are precisely those that preserve a certain notion of largeness
(\textquotedblleft central sets\textquotedblright ). In \cite{key-24},
authors completely accoplised the determination of image partition
regular matrices over $\mathbb{Z}\left[i\right]$.

Furstenberg and Glasner have shown in \cite{key-29} that for a particular
notion of largeness in a group, namely piecewise syndeticity, if a
set $B$ is a large subset $\mathbb{Z}$, then for any $l\in\mathbb{N}$,
the set of length $l$ arithmetic progressions lying entirely in $B$
is large among the set of all length $l$-aritmetic progressions.
In \cite{key-8}, authors extend this result to apply to infinitely
many notions of largeness in arbitrary semigroups and to other partition
regular structures besides arithmetic progressions. In \cite{key-43},
authors showed the abundance of image parition regular matrices .Also,
they examined for other well known notions of largeness.( Like IP$^{*}$,$\Delta^{*}$,PS$^{*}$,
piecewise syndetic, thick etc.). To obtain these results, usual order
relation of $\mathbb{N}$ played an important role. We have extend
these results of abunndance of matrices on $\mathbb{Z}\left[i\right]$,
the ring of Gaussian integers. In $\mathbb{Z}\left[i\right]$, there
is no order relation. So, we need to derive this results using different
technique.

First we need to prove this lemma before proceed to the main results:
\begin{lem}
Let $A$ be a $u\times v$ image partition regular matrix over $\mathbb{Q}\left[i\right]$.
The following are equivalent: 

(a) There exists $\vec{s}\in\left(\mathbb{Q}\left[i\right]\right)^{v}$
such that $A\vec{s}=\vec{1}$. 

(b) There exists $l\in\mathbb{Z}\left[i\right]\setminus\left\{ 0\right\} $
such that, if $p$ is in the smallest ideal of $\beta\mathbb{Z}\left[i\right]$
and $l\mathbb{Z}\left[i\right]\in p$, then, for every $P\in p$,
there exists $\vec{z}\in\left(\mathbb{Z}\left[i\right]\right)^{v}$such
that $A\vec{z}\in P^{u}$.
\end{lem}

\begin{proof}
$\left(a\right)\Rightarrow\left(b\right)$. We can choose $\vec{w}\in\left(\mathbb{Z}\left[i\right]\right)^{v}$
and $l\in\mathbb{Z}\left[i\right]\setminus\left\{ 0\right\} $ such
that $A\vec{w}=\vec{l}$, where $\vec{l}=\left(l,l,\ldots,l\right)^{T}\in\left(\mathbb{Z}\left[i\right]\setminus\left\{ 0\right\} \right)^{u}$.
Suppose that $p$ is in the smallest ideal of $\beta\mathbb{Z}\left[i\right]$
and that $l\mathbb{Z}\left[i\right]\in p$. Let $P\in p$. There exists
a minimal idempotent $q\in\beta\mathbb{Z}\left[i\right]$ such that
$p=p+q$ \cite[Theorem 2.8 and Lemma 1.30]{key-53}. Let $P'=\left\{ b\in P:-b+P\in q\right\} $.
Then $P'\in p$. Since $l\mathbb{Z}\left[i\right]\in p$ we can choose
$t\in\mathbb{Z}\left[i\right]$ such that $lt\in\mathbb{Z}\left[i\right]$.
Let $Q=-lt+P\in q$. By \cite[Theorem 3.1(h)]{key-24}, there exists
$\vec{x}\in\left(\mathbb{Z}\left[i\right]\setminus\left\{ 0\right\} \right)^{v}$
such that $A\vec{x}\in Q^{v}$ and for every $i\in\left\{ 1,2,\ldots,v\right\} $.
Then $t\vec{w}+\vec{x}\in\left(\mathbb{Z}\left[i\right]\setminus\left\{ 0\right\} \right)^{v}$
and $A\left(t\vec{w}+\vec{x}\right)=lt+A\vec{x}\in P^{u}$.

$\left(\text{b}\right)\Rightarrow\left(\text{a}\right)$. We may suppose
that the entries of $A$ are in $\mathbb{Z}\left[i\right]$ as we
could replace $A$ by $nA$ for a suitable $n\in\mathbb{N}$. 

Suppose that $\vec{1}\notin\left\{ A\vec{x}:\vec{x}\in\left(\mathbb{Q}\left[i\right]\right)^{v}\right\} $.
Then there exists $\vec{u}\in\left(\mathbb{Z}\left[i\right]\right)^{v}$
such that $\vec{u}\cdot A\vec{x}=0$ for every $\vec{x}\in\left(\mathbb{Q}\left[i\right]\right)^{v}$,
but $\vec{u}.\vec{1}\neq0$.

Choose a Gaussian prime number $r\in\mathbb{Z}\left[i\right]\setminus\left\{ 0\right\} $
satisfying $\left|r\right|>\left|l\right|$ and $\left|r\right|>\left|l\vec{u}\cdot\vec{1}\right|$.
Let $q$ be a minimal idempotent in $\beta\mathbb{Z}\left[i\right]$
and let $p=l+q$. Then $P=\left\{ t\in\mathbb{Z}\left[i\right]\setminus\left\{ 0\right\} :t\equiv l\left(\mod r\right)\right\} \in p$($t\equiv l\left(\mod r\right)\Rightarrow$$t-l$
is divisible by $r$ in $\mathbb{Z}\left[i\right]$) (by \cite[Lemma 2.1]{key-24}).
It follows from $\left(\text{b}\right)$ that there exists $\vec{z}\in\left(\mathbb{Z}\left[i\right]\right)^{v}$
such that $A\vec{z}\in P^{u}$ and hence that $\left(A\vec{z}\right)_{i}\equiv l\left(\mod r\right)$
for $i\in\left\{ 1,2,\ldots,u\right\} $. Thus, $A\vec{z}\equiv l\vec{1}\left(\mod\vec{r}\right)$,
a contradiction, as $\vec{u}\cdot A\vec{z}\equiv0\left(\mod r\right)$,
but $\vec{u}\cdot l\vec{1}\neq0\left(\mod r\right)$.
\end{proof}

\section{Large sets in $\mathbb{Z}\left[i\right]$}

There are several notions of largeness that make sense in any semigroup.
The notion of \textquotedblleft central\textquotedblright{} sets is
one of these. Among the others are the notions of \textquotedblleft syndetic\textquotedblright ,
\textquotedblleft piecewise syndetic\textquotedblright , \textquotedblleft IP\textquotedblright ,
and \textquotedblleft $\Delta$\textquotedblright{} sets.
\begin{defn}
Let $\left(S,+\right)$ be a commutative semigroup and let $B\subseteq S$. 

(a) The set $B$ is syndetic if and only if there exists some $G\in\mathcal{P}_{f}\left(S\right)=\left\{ H\subseteq S:H\text{ is finite and nonempty}\right\} $
such that $S=\cup_{t\in G}-t+B$. 

(b) The set $B$ is piecewise syndetic if and only if there exists
some $G\in\mathcal{P}_{f}\left(S\right)$ such that for every $F\in\mathcal{P}_{f}\left(S\right)$
there exists $x\in S$ such that $F+x\subseteq\cup_{t\in G}-t+B$.

(c) The set $B$ is an IP set if and only if there exists a sequence
$\left\langle x_{n}\right\rangle _{n=1}^{\infty}$ in $S$ such that
$FS\left(\left\langle x_{n}\right\rangle _{n=1}^{\infty}\right)\subseteq B$,
where $FS\left(\left\langle x_{n}\right\rangle _{n=1}^{\infty}\right)=\left\{ \sum_{n\in F}x_{n}:F\in\mathcal{P}_{f}\left(\mathbb{N}\right)\right\} $
and the sums are taken in increasing order of indices.

(d) The set $B$ is a $\Delta$ set if and only if there exists a
sequence $\left\langle x_{n}\right\rangle _{n=1}^{\infty}$ in $S$
such that for every $n,m\in\mathbb{N}$ with $n<m$, $x_{m}\in\left(x_{n}+B\right)$.
\end{defn}

Notice that, if $S$ can be embedded in a group $G$, $B$ is a $\Delta$
set if and only if there is a sequence $\left\langle x_{n}\right\rangle _{n=1}^{\infty}$
such that $\left\{ -x_{n}+x_{m}:m,n\in\mathbb{N}\text{ and }n<m\right\} \subseteq B$.
Notice also that any IP set is a $\Delta$ set. (Given $\left\langle x_{n}\right\rangle _{n=1}^{\infty}$
with $FS\left(\left\langle x_{n}\right\rangle _{n=1}^{\infty}\right)\subseteq B$
and given $n\in\mathbb{N}$, let $y_{n}=\sum_{t=1}^{n}x_{n}$). Given
any property $\mathcal{E}$ of subsets of a set $X$, there is a dual
property $\mathcal{E^{*}}$ defined by specifying that a subset $B$
of $X$ is an $\mathcal{E^{*}}$ set if and only if $B\cap A\neq\emptyset$
for every $\mathcal{E}$ set $A$.
\begin{defn}
Let $\left(S,+\right)$ be a commutative semigroup and let $B\subseteq S$.
Then B is a central$^{*}$ set if and only $B\cap A\neq\emptyset$
if for every central set $A$ in $S$. Also, $B$ is a PS$^{*}$ set
if and only if $B\cap A\neq\emptyset$ for every piecewise syndetic
set $A$ in $S$, $B$ is an IP$^{*}$ set if and only if $B\cap A\neq\emptyset$
for every IP set $A$ in $S$, $B$ is a syndetic$^{*}$ set if and
only if $B\cap A\neq\emptyset$ for every syndetic set $A$ in $S$,
and $B$ is a $\Delta^{*}$ set if and only if $B\cap A\neq\emptyset$
for every set $A$ in $S$.
\end{defn}

The concept of \textquotedblleft syndetic$^{*}$\textquotedblright{}
is more commonly referred to as \textquotedblleft thick\textquotedblright ,
and we shall follow this practice. 

The sets $\Delta$ and $\Delta^{*}$ sets are interesting because
they arise as sets of recurrence, which in turn have significant combinatorial
properties. (See \cite{key-28}.) The other notions discussed above
have simple, and useful, algebraic characterizations in terms of $\beta S$.
\begin{lem}
\label{Lemma 1.3}Let $\left(S,+\right)$ be a commutative semigroup
and let $B\subseteq S$. 

(a) $A$ is piecewise syndetic if and only if $\bar{A}\cap K\left(\beta S\right)\neq\emptyset$.

(b) $A$ is IP if and only if there is some idempotent of $\beta S$
in $\bar{A}$.

(c) $A$ is syndetic if and only if for every left ideal $L$ of $\beta S$,
$\bar{A}\cap L\neq\emptyset$.

(d) $A$ is central if and only if there is some minimal idempotent
of $\beta S$ in $\bar{A}$.

(e) $A$ is central$^{*}$ if and only if every minimal idempotent
of $\beta S$ is in $\bar{A}$.

(f ) $A$ is thick if and only if $A$ contains a left ideal of $\beta S$.

(g) $A$ is IP$^{*}$ if and only if every idempotent of $\beta S$
is in $\bar{A}$.

(h) $A$ is PS$^{*}$ if and only if $K\left(\beta S\right)\subseteq\bar{A}$.
\end{lem}

\begin{proof}
Statement (a) is \cite[Theorem 4.40]{key-53}, (b) is \cite[Theorem 5.12]{key-53},
(c) is \cite[Theorem 2.9(d)]{key-13}, and (d) is the definition of
central. Statements (e), (f), (g), and (h) follow easily from statements
(d), (c), (b), and (a) respectively. 
\end{proof}
Now we will prove a lemma we need.
\begin{lem}
\label{l=00003D1}Let $A$ be a $u\times$v image partition regular
matrix over $\mathbb{Q}\left[i\right]$. The following are equivalent: 

(a) There exists $\vec{s}\in\mathbb{Q}\left[i\right]^{v}$ such that
$A\vec{s}=\vec{1}$.

(b) There exists $l\in\mathbb{Z}\left[i\right]\setminus\left\{ 0\right\} $
such that, if $p$ is in the smallest ideal of $\beta\mathbb{Z}\left[i\right]$
and $l\mathbb{Z}\left[i\right]\in p$, then, for every $P\in p$,
there exists $\vec{z}\in\left(\mathbb{Z}\left[i\right]\setminus\left\{ 0\right\} \right)$
such that $A\vec{z}\in P^{u}$.
\end{lem}

\begin{proof}
(a) $\Rightarrow$ (b). We can choose $\vec{u}\in\left(\mathbb{Z}\left[i\right]\setminus\left\{ 0\right\} \right)^{v}$
and $l\in\mathbb{Z}\left[i\right]\setminus\left\{ 0\right\} $ such
that $A\vec{z}=\vec{l}$, where $\vec{l}=\left(ll\ldots l\right)^{T}\in\left(\mathbb{Z}\left[i\right]\setminus\left\{ 0\right\} \right)^{u}$.
Suppose that $p$ is in the smallest ideal of $\beta\mathbb{Z}\left[i\right]$
and that $l\mathbb{Z}\left[i\right]\in p$. Let $P\in p$. There exists
a minimal idempotent $q\in\beta\mathbb{Z}\left[i\right]$ such that
$p=p+q$ \cite[Theorem 2.8 and Lemma 1.30]{key-53}. Let $P'=\left\{ v\in P:-v+P\in q\right\} $.
Then $P'\in p$. Since $l\mathbb{Z}\left[i\right]\in p$ we can choose
$m\in\mathbb{Z}\left[i\right]\setminus\left\{ 0\right\} $ such that
$lm\in P'$. Let $Q=-lm+P\in q$. By \cite[Theorem 3.1(h)]{key-24},
there exists $\vec{z}\in\left(\mathbb{Z}\left[i\right]\setminus\left\{ 0\right\} \right)^{v}$
such that $A\vec{z}\in Q^{v}$ and $\left|z_{i}+mu_{i}\right|>0$
for every $i\in\left\{ 1,2,\ldots,v\right\} $. (The fact that the
entries of $\vec{z}$ can be chosen to be arbitrarily large follows
from \cite[Lemma 2.3]{key-24}. For every $r\in\mathbb{N}$, $\left\{ \vec{x}\in\left(\mathbb{Z}\left[i\right]\setminus\left\{ 0\right\} \right)^{v}:\left|x_{i}\right|>r\text{ for all }i\in\left\{ 1,2,\ldots,v\right\} \right\} $
is a member of every idempotent in $\beta\left(\mathbb{Z}\left[i\right]^{v}\right)$.)
Then $m\vec{z}+\vec{x}\in\left(\mathbb{Z}\left[i\right]\setminus\left\{ 0\right\} \right)^{v}$
and $A\left(m\vec{z}+\vec{x}\right)=m\vec{l}+A\vec{x}\in P^{u}$.

(b) $\Rightarrow$ (a). We may suppose that the entries of $A$ are
in $\mathbb{Z}\left[i\right]$, as we could replace $A$ by $nA$
for a suitable $n\in\mathbb{N}$. 

Suppose that $\vec{1}\notin\left\{ A\vec{z}:\vec{z}\in\mathbb{Q}\left[i\right]^{v}\right\} $.
Then there exists $\vec{u}\in\mathbb{Z}^{v}$ such that $\vec{u}\cdot A\vec{z}=0$
for every $\vec{z}\in\mathbb{Q}\left[i\right]^{v}$, but $\vec{u}\cdot\vec{1}\neq0$.

Choose a prime number $r$ satisfying $r>\left|l\right|$ and $r>\left|l\vec{u}\cdot\vec{1}\right|$.
Let $q$ be a minimal idempotent in $\beta\mathbb{Z}\left[i\right]$
and let $p=l+q$. Then 
\[
P=\left\{ v\in\mathbb{Z}\left[i\right]\setminus\left\{ 0\right\} :v\equiv l\left(\mod r\right)\right\} \in p\ \ \ \ \ \text{ (by [Lemma 2.14]).}
\]
 It follows from (b) that there exists $\vec{z}\in\mathbb{Z}\left[i\right]$
such that $A\vec{z}\in P^{u}$ and hence that $A\vec{z}\equiv l\vec{1}\,\left(\mod r\right)$
. This is a contradiction, as $\vec{u}\cdot A\vec{x}\equiv0\,\left(\mod r\right)$,
but $\vec{u}\cdot l\vec{1}$ is not divisible by $r$.
\end{proof}
The following theorem relates image partition regular matrices and
piecewise syndetic sets.
\begin{thm}
Let $A$ be a $u\times v$ image partition regular matrix over $\mathbb{Q}\left[i\right]$.
The following statements are equivalent:

(a) For every piecewise syndetic subset $P$ of $\mathbb{Z}\left[i\right]\setminus\left\{ 0\right\} $,
there exists $\vec{z}\in\left(\mathbb{Z}\left[i\right]\setminus\left\{ 0\right\} \right)^{v}$
such that $A\vec{z}\in P^{u}$;

(b) There exists $\vec{z}\in\left(\mathbb{Z}\left[i\right]\right)^{v}$such
that $A\vec{z}=\vec{1}$, where $\vec{1}$ denotes the vector in $\left(\mathbb{Z}\left[i\right]\right)^{u}$
whose entries are all equal to $1$.
\end{thm}

\begin{proof}
This follows easily from Lemma \ref{Lemma 1.3}(a) and the proof of
(a) $\Leftrightarrow$ (b) in Lemma \ref{l=00003D1}, taking $l=1$.
\end{proof}
In \cite{key-9} it was shown that if \textquotedblleft large\textquotedblright{}
meant any of \textquotedblleft $\Delta$\textquotedblright , \textquotedblleft IP\textquotedblright ,
\textquotedblleft central\textquotedblright , \textquotedblleft central$^{*}$\textquotedblright ,
\textquotedblleft IP$^{*}$\textquotedblright , or \textquotedblleft $\Delta^{*}$\textquotedblright ,
and $B$ is a large subset of $\mathbb{N}$, then for every positive
$\alpha\in\mathbb{R}$ and every $\gamma\in\mathbb{R}$ with $0<\gamma<1$,
$\left\{ \left\lfloor \alpha n+\gamma\right\rfloor :n\in B\right\} $
is also large (in the same sense). In \cite{key-29} it was shown
that if $B$ is a piecewise syndetic subset of $\mathbb{Z}$,$l\in\mathbb{N}$,
and $AP^{l}=\left\{ \left(a,a+d,\ldots,a+\left(l-1\right)d\right):a,d\in\mathbb{Z}\right\} $,
the group of length $l$ arithmetic progressions (including the constant
ones), then $B^{l}\cap AP^{l}$ is piecewise syndetic in $AP^{l}$.
In \cite{key-8} a systematic study of this latter phenomenon was
undertaken. These results apply in the current context in terms of
when the set of images contained in a given set is large among the
set of all images.
\begin{thm}
\label{Theorem 1.5}Let $A$ be a $u\times v$ matrix with entries
from $\mathbb{Q}\left[i\right]$, let $I=\left\{ A\vec{z}:\vec{z}\in\left(\mathbb{Z}\left[i\right]\right)^{v}\right\} \cap\left(\mathbb{Z}\left[i\right]\right)^{u}$,
and let $C\subseteq\mathbb{Z}\left[i\right]$.

(a) If $I\neq\emptyset$, \textquotedblleft large\textquotedblright{}
is any of \textquotedblleft IP$^{*}$\textquotedblright , \textquotedblleft $\Delta^{*}$\textquotedblright ,
\textquotedblleft PS{*}\textquotedblright , or \textquotedblleft central{*}\textquotedblright ,
and $C$ is large in $\mathbb{Z}\left[i\right]$, then $I\cap C^{u}$
is large in $I$.

(b) If $\bar{1}\in I$, \textquotedblleft large\textquotedblright{}
is any of \textquotedblleft piecewise syndetic\textquotedblright ,
\textquotedblleft central\textquotedblright , or \textquotedblleft thick\textquotedblright ,
and $C$ is large in $\mathbb{Z}\left[i\right]$, then $I\cap C^{u}$
is large in $I$.
\end{thm}

\begin{proof}
(a) For IP{*} and $\Delta${*}, \cite[Corollary 2.3]{key-8} requires
only that $I$ be a subsemigroup of $\left(\mathbb{Z}\left[i\right]\right)^{u}$.
For PS{*} and central{*}, \cite[Corollary 2.7]{key-8} requires in
addition that for each $i\in\left\{ 1,2,\ldots,u\right\} $, the $i$-th
projection $\pi_{i}\left[I\right]$ be piecewise syndetic in $\mathbb{Z}\left[i\right]$.
This trivially holds because, if $z\in\pi_{i}\left[I\right]$, then
$z\mathbb{Z}\left[i\right]\in\pi_{i}\left[I\right]$. 

(b) Letting $E=I$, we have that 
\[
\left\{ \begin{array}{c}
a\\
a\\
\vdots\\
a
\end{array}:a\in\mathbb{Z}\left[i\right]\right\} \subseteq E
\]
 so that \cite[Theorem 3.7]{key-8} applies.
\end{proof}
We shall be concerned for the rest of this section with establishing
analogues of \cite[Theorem 3.1(i)]{key-24} for the other notions
of largeness. That is, we wish to determine conditions that guarantee
that if a set $C$ is \textquotedblleft large\textquotedblright{}
in $\mathbb{Z}\left[i\right]$, then $\left\{ \vec{x}\in\left(\mathbb{Z}\left[i\right]\right)^{v}:A\vec{x}\in C^{u}\right\} $
is \textquotedblleft large\textquotedblright{} in $\left(\mathbb{Z}\left[i\right]\right)^{v}$.
\begin{lem}
\label{Lemma 1.6}Let $A$ be a $u\times v$ matrix with entries from
$\mathbb{Z}\left[i\right]$, define $\varphi:\left(\mathbb{Z}\left[i\right]\right)^{v}\rightarrow\left(\mathbb{Z}\left[i\right]\right)^{u}$
by $\varphi\left(x\right)=A\vec{x}$, and let $\tilde{\varphi}:\beta\left(\mathbb{Z}\left[i\right]^{v}\right)\rightarrow\left(\beta\mathbb{Z}\left[i\right]\right)^{u}$
be its continuous extension. Then $\tilde{\varphi}$ is a homomorphism
and $K\left(\left(\beta\mathbb{Z}\left[i\right]\right)^{u}\right)=\left(K\left(\beta\mathbb{Z}\left[i\right]\right)\right)^{u}$.

(a) If there exists $\vec{z}\in\left(\mathbb{Z}\left[i\right]\right)^{v}$
such that $A\vec{z}\in\left(\mathbb{Z}\left[i\right]\right)^{u}$,
then $\tilde{\varphi}\left[\beta\left(\mathbb{Z}\left[i\right]\right)^{v}\right]\cap K\left(\left(\beta\mathbb{Z}\left[i\right]\right)\right)^{u}\neq\emptyset$.

(b) If for all $\vec{z}\in\mathbb{Z}\left[i\right]^{v}$, $A\vec{z}\in\mathbb{Z}\left[i\right]^{u}$,
then $\tilde{\varphi}\left[K\left(\beta\left(\mathbb{Z}\left[i\right]^{v}\right)\right)\right]\subseteq K\left(\left(\beta\mathbb{Z}\left[i\right]\right)\right)^{u}$. 
\end{lem}

\begin{proof}
By \cite[Corollary 4.22]{key-53} we have that $\tilde{\varphi}$
is a homomorphism, and by \cite[Theorem 2.23]{key-53} $K\left(\left(\beta\mathbb{Z}\left[i\right]\right)^{u}\right)=\left(K\left(\beta\mathbb{Z}\left[i\right]\right)\right)^{u}$. 

(a). Since $\tilde{\varphi}\left[\beta\left(\mathbb{Z}\left[i\right]\right)^{v}\right]=\overline{\left\{ A\vec{z}:\vec{z}\in\left(\mathbb{Z}\left[i\right]\right)^{v}\right\} }$,
we need to show that 
\[
\overline{\left\{ A\vec{z}:\vec{z}\in\left(\mathbb{Z}\left[i\right]\right)^{v}\right\} }\cap\left(K\left(\beta\mathbb{Z}\left[i\right]\right)\right)^{u}\neq\emptyset.
\]
Pick $\vec{z}\in\mathbb{Z}\left[i\right]^{v}$ such that $A\vec{z}\in\mathbb{Z}\left[i\right]^{u}$.
Pick any minimal idempotent $p$ in $\beta\mathbb{Z}\left[i\right]$.
Then by Lemma \cite[Lemma 2.2]{key-24} $\vec{p}=\vec{y}p\in\left(K\left(\beta\mathbb{Z}\left[i\right]\right)\right)^{u}$.
To see that $\vec{p}\in\overline{\left\{ A\vec{w}:\vec{w}\in\mathbb{Z}\left[i\right]^{v}\right\} }$,
let $U$ be a neighborhood of $\vec{p}$ and for each $i\in\left\{ 1,2,\ldots,u\right\} $,
pick $D_{i}\in p$ such that $\times_{i=1}^{u}\overline{y_{i}D_{i}}\subseteq U$.
Pick $a\in\bigcap_{i=1}^{u}D_{i}$. Then $A\left(a\vec{z}\right)=a\vec{y}\in U$.

(b). By part (a), $\tilde{\varphi}\left[\beta\left(\mathbb{Z}\left[i\right]\right)^{v}\right]\cap K\left(\left(\beta\mathbb{Z}\left[i\right]\right)\right)^{u}\neq\emptyset$,
so by \cite[Theorem 1.65]{key-53}, $K\left(\tilde{\varphi}\left[\beta\left(\mathbb{Z}\left[i\right]\right)^{v}\right]\right)=\tilde{\varphi}\left[\beta\left(\mathbb{Z}\left[i\right]\right)^{v}\right]\cap K\left(\left(\beta\mathbb{Z}\left[i\right]\right)\right)^{u}$
(because $\tilde{\varphi}\left[\beta\left(\mathbb{Z}\left[i\right]^{v}\right)\right]\subseteq\left(\beta\mathbb{Z}\left[i\right]\right)^{u}$).
Also by \cite[Exercise 1.7.3]{key-53}, $K\left(\tilde{\varphi}\left[\beta\left(\mathbb{Z}\left[i\right]^{v}\right)\right]\right)=\tilde{\varphi}\left[K\left(\beta\left(\mathbb{Z}\left[i\right]^{v}\right)\right)\right]$.
\end{proof}
\begin{thm}
\label{Theorem 1.7}Let $A$ be a $u\times v$ matrix with entries
from $\mathbb{Q}\left[i\right]$ and assume that for all $\vec{z}\in\mathbb{Z}\left[i\right]^{v}$,
every entry of $A\vec{z}$ has both positive real and imaginary parts.
If \textquotedblleft large\textquotedblright{} is any of \textquotedblleft IP{*}\textquotedblright ,
\textquotedblleft $\Delta${*}\textquotedblright , or \textquotedblleft central{*}\textquotedblright ,
and $C$ is large in $\mathbb{Z}\left[i\right]$, then $W=\left\{ \vec{z}\in\left(\mathbb{Z}\left[i\right]\setminus\left\{ 0\right\} \right)^{v}:A\vec{z}\in C^{u}\right\} $
is large in $\mathbb{Z}\left[i\right]^{v}$.
\end{thm}

\begin{proof}
We show first that it suffices to prove the theorem under the additional
assumption that all entries of $A$ are in $\mathbb{Z}\left[i\right]$.
Indeed, suppose we have done so and pick $d\in\mathbb{N}$ such that
all entries of $nA$ are in $\mathbb{Z}\left[i\right]$. We claim
that $dC$ is large in $\mathbb{Z}\left[i\right]$, which we check
individually.

Assume first that \textquotedblleft large\textquotedblright{} is \textquotedblleft IP{*}\textquotedblright{}
and let a sequence $\left\langle z_{n}\right\rangle _{n=1}^{\infty}$
in $\mathbb{Z}\left[i\right]$ be given. By \cite[Theorem 5.14]{key-53}
and \cite[Lemma 2.1]{key-24} pick a sum subsystem $\left\langle w_{n}\right\rangle _{n=1}^{\infty}$
of $\left\langle z_{n}\right\rangle _{n=1}^{\infty}$ such that $z\mid y_{n}$
for each $n$. Pick $a\in C\cap FS\left(\left\langle \frac{y_{n}}{z}\right\rangle _{n=1}^{\infty}\right)$.
Then $da\in dC\cap FS\left(\left\langle z_{n}\right\rangle _{n=1}^{\infty}\right)$.

Next assume that \textquotedblleft large\textquotedblright{} is \textquotedblleft $\Delta${*}\textquotedblright .
Let a set $B$ in $\mathbb{Z}\left[i\right]$ be given and choose
a sequence $\left\langle z_{n}\right\rangle _{n=1}^{\infty}\subseteq\mathbb{Z}\left[i\right]\setminus\left\{ 0\right\} $
such that for every $n,m\in\mathbb{N}$ with $n<m$, $z_{m}\in z_{n}+B$.
(In particular, for each $n<m,\ \left|z_{n}\right|<\left|z_{m}\right|$.)
By passing to a subsequence, we may presume that for each $n<m$,
$z_{n}\equiv z_{m}\left(\text{mod }d\right)$. Pick $j\in\left\{ 1,2,\ldots,d\right\} $
such that for each $n\in\mathbb{N}$, $z_{n}+j\equiv0\left(\text{mod }d\right)$.
Then $\left\langle \frac{z_{n}+j}{d}\right\rangle _{n=1}^{\infty}$
is a sequence in $\mathbb{Z}\left[i\right]\setminus\left\{ 0\right\} $
so pick $n<m$ such that $\frac{z_{m}+j}{d}-\frac{z_{n}+j}{d}\in C$.
Then $z_{m}-z_{n}\in dC$. 

Finally assume that \textquotedblleft large\textquotedblright{} is
\textquotedblleft central{*}\textquotedblright . Let $p$ be a minimal
idempotent in $\beta\mathbb{Z}\left[i\right]$. By Lemma \cite[Lemma 2.2]{key-24},
$\frac{1}{d}p$ is a minimal idempotent so $C\in\frac{1}{d}p$ and
consequently $dC\in p$. Since we have established that $dC$ is large,
we have that $\left\{ \vec{z}\in\left(\mathbb{Z}\left[i\right]\setminus\left\{ 0\right\} \right)^{v}:dA\vec{z}\in\left(dC\right)^{u}\right\} $
is large, and 
\[
\left\{ \vec{z}\in\left(\mathbb{Z}\left[i\right]\setminus\left\{ 0\right\} \right)^{v}:dA\vec{z}\in\left(dC\right)^{u}\right\} =\left\{ \vec{z}\in\left(\mathbb{Z}\left[i\right]\setminus\left\{ 0\right\} \right)^{v}:A\vec{z}\in C^{u}\right\} 
\]
.

Thus we assume that all entries of $A$ are in $\mathbb{Z}\left[i\right]$.
Define $\varphi:\mathbb{Z}\left[i\right]^{v}\rightarrow\mathbb{Z}\left[i\right]^{u}$
by $\varphi\left(\vec{z}\right)=A\vec{z}$, and let $\tilde{\varphi}:\beta\left(\mathbb{Z}\left[i\right]^{v}\right)\rightarrow\left(\beta\mathbb{Z}\left[i\right]\right)^{u}$
be its continuous extension. Assume first that \textquotedblleft large\textquotedblright{}
is \textquotedblleft IP{*}\textquotedblright . Let $p$ be an idempotent
in $\beta\left(\mathbb{Z}\left[i\right]^{v}\right)$. We need to show
that $W\in p$. Since $\tilde{\varphi}$ is a homomorphism, $\tilde{\varphi}\left(p\right)$
is an idempotent in $\left(\beta\mathbb{Z}\left[i\right]\right)^{u}$
and so $\bar{C}^{u}$ is a neighborhood of $\tilde{\varphi}\left(p\right)$,
and hence $W\in p$ as required.

Next assume that \textquotedblleft large\textquotedblright{} is \textquotedblleft $\Delta${*}\textquotedblright .
Let $B$ be a set in $\mathbb{Z}\left[i\right]^{v}$ and pick a sequence
$\left\langle \vec{z}_{n}\right\rangle _{n=1}^{\infty}$ in $\left(\mathbb{Z}\left[i\right]\setminus\left\{ 0\right\} \right)^{v}$
such that for every $n,m\in\mathbb{N}$ with $n<m$, $\vec{z_{m}}\in\vec{z_{n}}+B$.
In particular, for each $n<m$, we have $\vec{z_{m}}-\vec{z_{n}}\in\left(\mathbb{Z}\left[i\right]\setminus\left\{ 0\right\} \right)^{v}$.
We need to show that there exists $n<m$ such that $\vec{z_{m}}-\vec{z_{n}}\in W$. 

For each $n\in\mathbb{N}$, let $\vec{y_{n}}=A\vec{z_{n}}$. Notice
that, for $n<m$ we have that all real and imaginary entries of $A\left(\vec{z_{m}}-\vec{z_{n}}\right)$
are positive, and consequently modulus of each entry of $\vec{y_{m}}$
is larger than the modulus of corresponding entry of $\vec{y_{n}}$.
By Ramsey\textquoteright s Theorem (\cite[Theorem 1.5]{key-32}),
pick an infinite subset $D_{1}$ of $\mathbb{N}$ such that for all
$n<m$ in $D_{1}$, $y_{m}^{1}-y_{n}^{1}\in C$ or for all $n<m$
in $D_{1}$, $y_{m}^{1}-y_{n}^{1}\in\left(\mathbb{Z}\left[i\right]\setminus\left\{ 0\right\} \right)\setminus C$.
Since $C$ is a $\Delta${*} set, the latter alternative is impossible,
so the former must hold. Inductively, given $i\in\left\{ 1,2,\ldots,u-1\right\} $,
choose by Ramsey\textquoteright s Theorem an infinite subset $D_{i+1}$
of $D_{i}$ such that for all $n<m$ in $D_{i+1}$, $y_{m}^{i+1}-y_{n}^{i+1}\in C$.
Having chosen $D_{u}$ pick $n<m$ in $D_{u}$. Then $\vec{z_{m}}-\vec{z_{n}}\in W$.
Finally assume that \textquotedblleft large\textquotedblright{} is
\textquotedblleft central{*}\textquotedblright , and let $p$ be a
minimal idempotent in $\beta\left(\mathbb{Z}\left[i\right]^{v}\right)$.
By Lemma \ref{Lemma 1.6}, $\tilde{\varphi}\left(p\right)$ is an
idempotent and $\tilde{\varphi}\left(p\right)\in\left(K\left(\beta\mathbb{Z}\left[i\right]\right)\right)^{u}$
so that $\tilde{\varphi}\left(p\right)\bar{C}^{u}$. 
\end{proof}
The requirement in Theorem \ref{Theorem 1.7} that for all , every
$\vec{z}\in\mathbb{Z}\left[i\right]^{v}$ entry of $A\vec{z}$ be
positive may not be omitted ( see \cite{key-43}).
\begin{thm}
Let $A$ be a $u\times v$ matrix with entries from $\mathbb{Z}\left[i\right]$
and assume that for all $\vec{z}\in\left(\mathbb{Z}\left[i\right]\setminus\left\{ 0\right\} \right)^{v}$,
every entry of $A\vec{z}\in\left(\mathbb{Z}\left[i\right]\setminus\left\{ 0\right\} \right)^{u}$.
If $C$ is PS{*} in $\mathbb{Z}\left[i\right]$, then $W=\left\{ \vec{z}\in\left(\mathbb{Z}\left[i\right]\setminus\left\{ 0\right\} \right)^{v}:A\vec{z}\in C^{u}\right\} $
is PS{*} in $\mathbb{Z}\left[i\right]^{v}$. 
\end{thm}

\begin{proof}
Define $\varphi:\left(\mathbb{Z}\left[i\right]\setminus\left\{ 0\right\} \right)^{v}\rightarrow\left(\mathbb{Z}\left[i\right]\setminus\left\{ 0\right\} \right)^{v}$
by $\varphi\left(\vec{z}\right)=A\vec{z}$, and let $\tilde{\varphi}:\beta\left(\mathbb{Z}\left[i\right]^{v}\right)\rightarrow\left(\beta\mathbb{Z}\left[i\right]\right)^{u}$
be its continuous extension. Let $p\in K\left(\beta\left(\mathbb{Z}\left[i\right]^{v}\right)\right)$.
By Lemma \ref{Lemma 1.6}, $\tilde{\varphi}\left(p\right)\in\left(K\left(\beta\mathbb{Z}\left[i\right]\right)\right)^{u}$
and thus $\tilde{\varphi}\left(p\right)\in\bar{C}^{u}$ and thus $W\in p$
as required.
\end{proof}
\begin{thm}
Let A be a $u\times v$ image partition regular matrix with entries
from $\mathbb{Z}\left[i\right]$. If $C$ is thick in $\mathbb{Z}\left[i\right]$,
then $W=\left\{ \vec{z}\in\left(\mathbb{Z}\left[i\right]\setminus\left\{ 0\right\} \right)^{v}:A\vec{z}\in C^{u}\right\} $
is thick in $\mathbb{Z}\left[i\right]^{v}$. 
\end{thm}

\begin{proof}
Since $C$ is thick, pick a left ideal $L$ of $\beta\mathbb{Z}\left[i\right]$
such that $L\subseteq\bar{C}$. Pick by \cite[Corollary 2.6]{key-53}
a minimal idempotent $p\in L$. Define $\varphi:\mathbb{Z}\left[i\right]^{v}\rightarrow\mathbb{Z}\left[i\right]^{u}$
by $\varphi\left(\vec{z}\right)=A\vec{z}$, and let $\tilde{\varphi}:\beta\left(\mathbb{Z}\left[i\right]^{v}\right)\rightarrow\left(\beta\mathbb{Z}\left[i\right]\right)^{u}$
be its continuous extension. Let $\bar{p}=\left(p,p,\ldots p\right)^{T}$
and pick by \cite[Lemma 2.3]{key-24} a minimal idempotent $q\in\beta\left(\mathbb{Z}\left[i\right]^{v}\right)$
such that $\tilde{\varphi}\left(q\right)=p$. (By \cite[Theorem 3.1(h)]{key-24},
$p$ satisfies the hypotheses of \cite[Theorem 2.3]{key-24}.) We
claim that $\tilde{\varphi}\left[\beta\left(\mathbb{Z}\left[i\right]^{v}\right)+q\right]\subseteq\bar{C}^{u}$
so that $\beta\left(\mathbb{Z}\left[i\right]^{v}\right)+q\subseteq\bar{W}$
as required. To this end, let $r\in\beta\left(\mathbb{Z}\left[i\right]^{v}\right)$
and let $i\in\left\{ 1,2,\ldots,u\right\} $. Then $\pi_{i}\circ\tilde{\varphi}\left(r+q\right)=\pi_{i}\left(\tilde{\varphi}\left(r\right)\right)+p\in\beta\mathbb{Z}\left[i\right]+p\subseteq L\subseteq\bar{C}$.
\end{proof}

\end{document}